\documentclass{article}
\usepackage{graphicx} % Required for inserting images

\usepackage{amsmath,amssymb,amsthm}
\usepackage[margin=2cm]{geometry}

\usepackage{xcolor}
\usepackage{hyperref}
\usepackage{authblk}

\usepackage{tikz}
\usetikzlibrary{positioning,arrows.meta}

\newtheorem{theorem}{Theorem}[section]

\newtheorem{remark}[theorem]{Remark}
\newtheorem{lemma}[theorem]{Lemma}

\newtheorem{proposition}[theorem]{Proposition}
\newtheorem{corollary}[theorem]{Corollary}

\DeclareMathOperator{\spann}{span}

\DeclareMathOperator{\sign}{sign}
\DeclareMathOperator{\wt}{wt}
\DeclareMathOperator{\ddp}{dp}
\DeclareMathOperator{\op}{op}

\title{On uniform eventowns}
\author[1]{Danila Cherkashin}
\author[2]{Pavel Prozorov}

\affil[1]{Institute of Mathematics and Informatics, Bulgarian Academy of Sciences, Sofia}
\affil[2]{Saint Petersburg State University, Saint Petersburg, Russia}

\date{July 2026}

\begin{document}

\maketitle

\begin{abstract}
  Suppose that $n=2m$, $k=2t$ and $n > 10 k^7$.
  We show that if a family $\mathcal F$ of $k$-subsets of an $n$-set has only even pairwise intersections then $|\mathcal F| \leq \binom{m}{t}$. 
  Moreover, every extremal family has an atomic structure.

  This result was previously proved by Frankl and Tokushige for $n > n_{FT}(k)$, where $n_{FT}(k)$ is at least exponential. 
  The main technique is Delsarte linear programming.
\end{abstract}

\section{Introduction}

\subsection{The eventown problem}

Let $[n]=\{1,\dots,n\}$.  A family $\mathcal E\subseteq 2^{[n]}$ is
called an \emph{eventown} if
\[
    |E\cap E'|\equiv 0\pmod 2
    \qquad\text{for all }E,E'\in\mathcal E.
\]
Allowing $E=E'$ in this condition ensures, in particular, that every member
of $\mathcal E$ has even cardinality.  If
\(
[n]=p_1\sqcup\cdots\sqcup p_{\lfloor n/2\rfloor}\sqcup R
\)
with $|p_i|=2$ and $|R|\leq1$, then the following so-called \textit{atomic construction}
\[
\mathcal{A} = \left\{
        \bigcup_{i\in I}p_i:
        I\subseteq[\lfloor n/2\rfloor]
    \right\}
\]
is an eventown of size $2^{\lfloor n/2\rfloor}$. Berlekamp~\cite{Berlekamp1969} and Graver~\cite{Graver1975}
proved independently that an atomic construction has the largest possible size:
\begin{equation*}
    |\mathcal E|\leq 2^{\lfloor n/2\rfloor}.
\end{equation*}
Their argument is one of the standard early examples of the linear-algebra
method in extremal set theory.  Indeed,
identifying a set $E$ with its incidence vector $v_E\in\mathbb F_2^n$ and
putting
\[
    C= \spann_{\mathbb F_2}\{v_E:E\in\mathcal E\},
\]
the eventown condition gives $C\subseteq C^\perp$.  Hence
$\dim C\leq\lfloor n/2\rfloor$, and therefore
$|\mathcal E|\leq|C|\leq2^{\lfloor n/2\rfloor}$.  

Moreover, equality forces $\mathcal E=C$, so the extremal families are precisely maximal-dimensional
totally isotropic binary subspaces. In coding-theoretic language an extremal eventown for even $n$ is called a self-dual code of Type I~\cite{nebe2006self}. 
The study of the combinatorial structure of such codes is a classical but still relevant problem in coding theory. 
One of the main instruments of this study is the weight enumerator polynomial
\[
W_C (x,y) := \sum_{c \in C} x^{\wt (c)} y^{n-\wt (c)},
\]
where $\wt (c)$ is the number of non-zero coordinates in a codeword $c$.
Although strong restrictions on the weight enumerators of self-dual
codes are known~\cite{bouyuklieva2021self}, complete weight distributions remain unknown for
many important families of codes, in particular for Reed--Muller codes~\cite{borissov2026verifying,markov2025weight}.

From the combinatorial point of view, much of the recent work has
focused on quantitative refinements and structural generalizations.  One
such direction is supersaturation.  For a family $\mathcal A$ of even-sized
subsets, write
\[
    \op(\mathcal A)
    :=
    \bigl|\{\{A,B\}\subseteq\mathcal A:
             |A\cap B|\text{ is odd}\}\bigr|.
\]
O'Neill initiated the problem of minimizing $\op(\mathcal A)$
when $|\mathcal A|$ exceeds the eventown bound, and proved the sharp estimate
\[
    \op(\mathcal A)
    \geq s\,2^{\lfloor n/2\rfloor-1}
\]
for $|\mathcal A|=2^{\lfloor n/2\rfloor}+s$ and $s\in\{1,2\}$
\cite{ONeill2023}.  Wei, Zhao, Zhang and Ge subsequently established the
same bound, for sufficiently large $n$, throughout the range
\(
1\leq s\leq 2^{\lfloor n/8\rfloor}/n
\), and obtained the general estimate
\(
\op(\mathcal A)
\geq s2^{\lfloor n/2\rfloor-2}
\)
by discrete Fourier analysis~\cite{WeiZhaoZhangGe2025}.  In a complementary
spectral direction, Antipov and Cherkashin recast eventown and several
modular variants in terms of the Lov\'asz theta function
\cite{AntipovCherkashin2022}.

A second active direction replaces pairwise parity conditions by restrictions
on higher intersections.  Sudakov and Vieira proved uniqueness and stability
results for $k$-wise eventowns when $k\geq3$~\cite{SudakovVieira2018}.
Johnston and O'Neill introduced the more general language of
$\boldsymbol\alpha$-intersection patterns, prescribing the parity of every
$i$-wise intersection for $1\leq i\leq k$, and determined the corresponding
extremal functions for all patterns of length three and asymptotically for
all patterns of length four~\cite{JohnstonONeill2025}.  This program was
further developed by Wei, Zhang and Ge~\cite{WeiZhangGe2025}; most recently,
Dong, Ouyang and Wei obtained sharp asymptotics for every fixed binary
intersection pattern~\cite{dong2026sharp}.  These results place the
classical eventown theorem inside a broader theory of parity-restricted set
systems.

\subsection{The uniform problem}

The present problem imposes a different restriction: all members are required
to lie in a single layer of the Boolean lattice.  This destroys the closure
under symmetric difference that makes the non-uniform theorem immediate,
since the span of incidence vectors is no longer confined to the prescribed
Hamming sphere.  The natural replacement for the ambient vector-space
structure is therefore the spectral theory of the Johnson association
scheme.

Let $[n]=\{1,\dots,n\}$, and let $t\geq 1$. 
A family $\mathcal F\subseteq \binom{[n]}{2t}$ is called a \emph{$2t$-uniform eventown} if
\[
    |f\cap f'|\equiv 0 \pmod 2
    \qquad\text{for all }f,f'\in\mathcal F.
\]
We denote its maximum possible size by
\[
    m_{\mathrm{ev}}(n,2t)
    :=
    \max\left\{
        |\mathcal F|:
        \mathcal F\subseteq\binom{[n]}{2t},\ 
        |f\cap f'|\equiv0\pmod2
        \text{ for all }f,f'\in\mathcal F
    \right\}.
\]
A $2t$-uniform \textit{atomic eventown} $\mathcal A_t$ is defined as the intersection of the atomic construction $\mathcal{A}$ with the $2t$-layer of the Hamming cube. From now on, assume that $n=2m$. Clearly 
\[
  m_{\mathrm{ev}}(n,2t)  \geq  |\mathcal A_t| = \binom{m}{t}.
\]

Frankl and Tokushige~\cite{FranklTokushige2016} proved that for every fixed $t$ and even $n > n_{FT}(k)$ the equality
\[
    m_{\mathrm{ev}}(n, 2t) = \binom{n/2}{t}
\]
holds.  
The quantitative ingredients for $n_{FT(k)}$ in their
structural proof operate at least at an exponential scale in the uniformity. 

For even $n = 2m$ they also proposed a different approach, based on the Delsarte--Hoffman ratio
bound in the Johnson association scheme, and carried it out explicitly for
$2t=6$, obtaining the sharp bound for $n\geq 26$.  The purpose of the
present argument is to formulate this linear-programming certificate for
arbitrary $t$ and to make its dependence on $t$ accessible to asymptotic
analysis.

The spectral problem has two parts.  First, one has to solve the
Frankl--Tokushige equal-eigenvalue equations in strictly positive variables
\[
    a_1,a_3,\dots,a_{2t-1}>0.
\]
Second, for the resulting pseudo-adjacency matrix one has to verify that the
trivial eigenvalue is the largest eigenvalue and that the common nontrivial
even eigenvalue is the least eigenvalue.  We shall refer to the least
threshold $n_{LP}(k)$ for which this particular certificate works for all larger even
$n$ as the \emph{LP threshold}.

Frankl and Tokushige provide numerical evidence that $n_{LP}(k)$ should be much smaller than $n_{FT}(k)$. Here we prove a polynomial upper bound for the quantity $n_{LP}(k)$.

\begin{theorem} \label{th:main}
    Let $n = 2m$, $k = 2t$ and $n > 10k^7$. Then
    \[
    m_{\mathrm{ev}} (2m,2t) = \binom{m}{t}.
    \]
    Moreover, equality is attained only by an atomic construction $\mathcal{A}_t$. 
\end{theorem}

Our main ingredient in the proof of the upper bound in
Theorem~\ref{th:main} is the sign pattern established in
Lemma~\ref{lm:key}.  In particular, all auxiliary variables arising
from the corresponding linear system have prescribed signs.  This
allows us to identify and estimate the leading terms of the relevant
determinants.  The second step is to prove that these estimates are
stable under the perturbations appearing in the original linear
system.

To prove uniqueness, we first determine the distance distribution of
an extremal family using the Delsarte linear-programming machinery.
We then embed the uniform eventown $\mathcal E_t$ into a non-uniform
eventown $\mathcal E$ by extending the linear span of its incidence
vectors to a maximal self-orthogonal binary code.  The distance
distribution yields a large $4$-uniform eventown $\mathcal E_2$ of a
specific form inside $\mathcal E$.  Since the uniformity $4$ is small
enough to permit a complete combinatorial classification, we can show
that $\mathcal E_2$ has the required atomic structure.  Finally, we
lift this structure from $\mathcal E_2$ to the original family
$\mathcal E_t$, completing the uniqueness proof.

\paragraph{Structure of the paper.}
Section~\ref{sec:basicsofLP} and Subsection~\ref{subsec:basicsofLP} contain the linear-programming
preliminaries used in the upper-bound and uniqueness parts of Theorem~\ref{th:main}, respectively.
Section~\ref{sec:proofOutline} gives an outline of the proof of the upper bound.  Section~\ref{sec:proof} contains the technical estimates
that form the main part of this proof.
Subsection~\ref{subsec:distdist} determines the distance distribution
of an extremal family, while Subsection~\ref{subsec:uniqueness} derives its atomic structure and
completes the uniqueness proof.

\section{The Delsarte linear programming approach in the Johnson scheme} \label{sec:basicsofLP}

For the standard theory of the Johnson scheme and its eigenmatrices,
see~\cite[Chapter~6]{GodsilMeagher2015Johnson}. Set $k=2t$ and
\[
    \mathcal X=\binom{[n]}{k}.
\]
For $0\leq j\leq k$, let $A_j$ be the matrix indexed by $\mathcal X\times \mathcal X$ with
\[
    (A_j)_{xy}=\mathbf 1_{\{|x\cap y|=j\}}.
\]
Let $G_{\mathrm{odd}}(n,k)$ be the graph on $\mathcal X$ in which two distinct
vertices are adjacent precisely when their intersection has odd size.
Every $2t$-uniform eventown is an independent set in this graph.

Following Frankl and Tokushige, consider the pseudo-adjacency matrix
\[
    M :=\sum_{s=1}^{t}a_{2s-1}A_{2s-1}.
\]
The matrices $A_0,\dots,A_k$ belong to the Bose--Mesner algebra of the
Johnson scheme and are simultaneously diagonalizable.  Write
\[
    \mathbb R^\mathcal X=V_0\oplus V_1\oplus\cdots\oplus V_k,
    \qquad
    V_0=\langle\mathbf 1\rangle,
\]
for the standard eigenspace decomposition. One has 
\[
\dim V_k = \binom{n}{k} - \binom{n}{k-1}.
\]

For every $i$ define
\[
x_i (S) = \begin{cases}
    1, \quad i \in S\\
    0, \quad i \notin S,
\end{cases}
\]
where $S \in \binom{[n]}{k}$. The following fact can be found for instance in~\cite{filmus2016orthogonal}.

\begin{lemma} \label{lm:eigenbase}
    The space $V_i$ is spanned by the functions $f_{\pi,i} : \binom{[n]}{k} \to \mathbb{R}$ defined as
    \[
    f_{\pi,i} := (x_{\pi(1)} - x_{\pi(2)}) (x_{\pi(3)} - x_{\pi(4)}) \dots (x_{\pi(2i-1)} - x_{\pi(2i)}),
    \]
    for all permutations $\pi \in S_n$.
\end{lemma}

Clearly every $f_{\pi,i}$ takes only values $0, \pm 1$.
It is also worth noting that functions $f_{\pi,i}$ have the smallest supports among all vectors in $V_i$ provided that $n > n_0(k,i)$, see~\cite{vorob2018minimum}.

The following celebrated theorem is widely known as Hoffman ratio bound~\cite{haemers2021hoffman}. We are going to apply it to $G = G_{\mathrm{odd}}$ and specific values of $A$.

\begin{theorem}
    Let $A$ be a symmetric pseudoadjacency matrix of an $N$-vertex graph $G$ with non-negative entries. Suppose that $(1,\dots,1)^T$ is an eigenvector of $A$. Then 
\begin{equation} \label{eq:Hoffmanita}
\alpha(G) \leq N \frac{-\lambda_{\min}}{\lambda_{\max} - \lambda_{\min}},
\end{equation}
where $\alpha(G)$ is the independence number of $G$.
\end{theorem}

The proof is a one-line collection of several simple observations. Let $I$ be any independent set in $G$ and $\chi_I$ stands for its characteristic vector. Then
\begin{equation} \label{eq:main}
0 = (A\chi_I,\chi_I) = \sum_{i=1}^N c_i^2 \lambda_i \geq c_1^2 \lambda_{\max} + (c_2^2 + \dots + c_N^2) \lambda_{\min} = \frac{|I|^2}{N} \lambda_{\max} + \left (|I| - \frac{|I|^2}{N}  \right) \lambda_{\min},
\end{equation}
where $c_i$ are the coefficients in the decomposition of $\chi_I$ in the eigenbasis of $A$.
Here we use that since all entries are non-negative, the largest eigenvalue $\lambda_{\max}$ is achieved at the all-ones vector.
Also in the case of an edge-transitive graph Hoffman bound coincides with Lov{\'a}sz theta-bound~\cite{lovasz1979shannon}.

The reverse-engineering condition proposed in
\cite{FranklTokushige2016} is to make all nontrivial even eigenvalues equal.
The remaining task is then to prove that this common value is the least
eigenvalue and that the resulting ratio is the density of the atomic
construction.

Let $P_e(j)$ denote the eigenvalue of $A_j$ on $V_e$.  In the present
intersection-size indexing, the dual Hahn polynomial (Eberlein polynomial) can be written as
\begin{equation}
    P_e(j)
    =
    \sum_{q=0}^{e}
    (-1)^{e-q}
    \binom{e}{q}
    \binom{k-q}{j-q}
    \binom{n-k+q-e}{k-j+q-e},
    \qquad 0\leq e,j\leq k,
    \label{eq:Hahn}
\end{equation}
where a binomial coefficient is understood to be zero whenever its lower
index is outside the usual range.  Consequently,
\begin{equation}
    \lambda_e
    =
    \sum_{s=1}^{t}a_{2s-1}P_e(2s-1),
    \qquad 0\leq e\leq 2t.
    \label{eq:eigenvalues}
\end{equation}

In the next section we show that this approach works for $n \geq 10k^7$. 

\section{Outline of the proof of the upper bound} \label{sec:proofOutline}

Recall that $G_{\mathrm{odd}}(n,k)$ is the graph on $\mathcal X = \binom{[n]}{k}$ in which two distinct
vertices are adjacent precisely when their intersection has odd size. We also write $n = 2m$, $k = 2t$.
We construct a pseudoadjacency matrix $M$ by the reverse-engineering method as a sum 
\[
    M=M(a):=\sum_{s=1}^{t}a_{2s-1}A_{2s-1}
\]
satisfying $\lambda_2 = \lambda_4 = \dots = \lambda_{2t}=-1$.

The following structural lemma is the key part of our proof. A long and technical proof is placed in Section~\ref{sec:proof}.

\begin{lemma}\label{lm:key}
    Let $n > 10k^7$. Then the following hold
    \begin{enumerate}
        \item the system
        \[
        \lambda_{2r}(a_1,a_3,\ldots,a_{2t-1})=-1, \qquad 1\le r\le t,
        \]
has a unique solution $a_1,a_3,\ldots,a_{2t-1}$;
        \item $a_{2i-1} > 0$ for every $1 \leq i \leq t$;
        \item $\lambda_{2i-1} > 0$ for every $1 \leq i \leq t$.
    \end{enumerate}
\end{lemma}
    
%Сетап такой: уравнения про лямбды и оны все равны -1. 

\begin{lemma}\label{lm:eigeneven}
   Let $n = 2m$. Then the characteristic vector $\chi_{\mathcal A_t}$ of an atomic family $\mathcal A_t$ belongs to
    \[
    V_0\oplus V_2\oplus\cdots\oplus V_{2t}; 
    \]
    i.e. to the direct sum of all even-indexed eigenspaces.
\end{lemma}

\begin{proof}
    By Lemma~\ref{lm:eigenbase} we need to show that 
    \[
    \langle \chi_{\mathcal A_t}, f_{\pi,i} \rangle = 0
    \]
    for every $\pi \in S_n$ and every odd $i$. 

    Now fix $\pi$ and $i$ and consider a graph $G(\pi,i)$ on the vertex set $[n]$ with $m$ red edges produced by the atoms of ${\mathcal A_t}$ and $i$ blue edges produced by the brackets of $f(\pi,i)$. The degree of every vertex in $G(\pi,i)$ is at most 2, so $G(\pi,i)$ is the union of paths and cycles (note that a $2$-cycle consisting of a red and a blue edge may occur). Since $i$ is odd, there is a connected component $C_{odd}$ of $G(\pi,i)$ which contains an odd number of blue edges. 
    Now choose a direction in the path or a cycle $C_{odd}$ and orient every blue edge according to this direction.
    Without loss of generality we can assume that every brackets $\pi$ subtracts the end of the corresponding blue edge from its origin.
    
    Now ``reflect'' this component $C_{odd}$ as shown in Figure~\ref{fig:path-cycle-reflection}.
    Such a reflection $\rho$ preserves both sets of red and blue edges, so 
    $f_{\pi,i} = \pm f_{\rho(\pi),i}$. 
    Since a reflection changes the orientation in the plane, $\rho$ changes sign of every bracket that corresponds to the component $C_{odd}$. This means that 
    \[
    \langle \chi_{\mathcal A_t}, f_{\pi,i} \rangle = \langle \chi_{\mathcal A_t}, -f_{\pi,i} \rangle
    \]
    and thus $\langle \chi_{\mathcal A_t}, f_{\pi,i} \rangle = 0$.

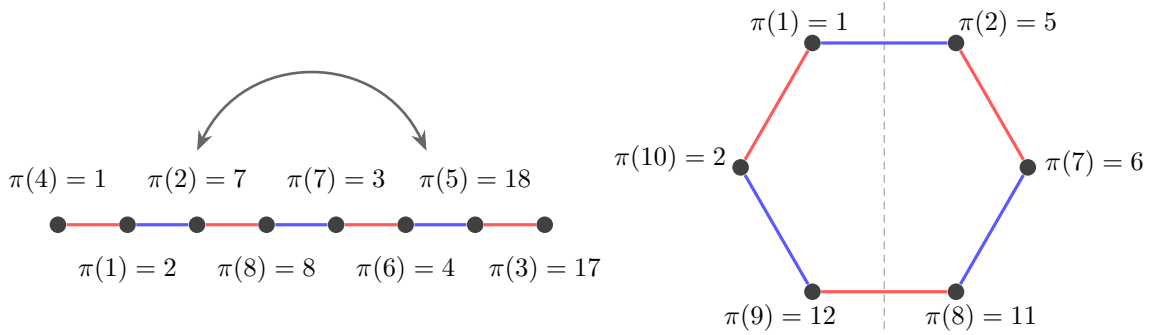
\begin{figure}[ht]
    \centering
       \begin{minipage}[c]{0.50\textwidth}
        %\centering
        \begin{tikzpicture}[
            x=0.92cm,
            y=0.85cm,
            point/.style={circle, fill=black!75, inner sep=2.2pt},
            rededge/.style={red!65, line width=1.2pt},
            blueedge/.style={blue!65, line width=1.2pt},
            label/.style={font=\normalsize},
            perm/.style={font=\normalsize}
        ]

        \node[point] (p1)  at (0,0) {};
        \node[point] (p2)  at (1,0) {};
        \node[point] (p7)  at (2,0) {};
        \node[point] (p8)  at (3,0) {};
        \node[point] (p3)  at (4,0) {};
        \node[point] (p4)  at (5,0) {};
        \node[point] (p18) at (6,0) {};
        \node[point] (p17) at (7,0) {};

        \draw[rededge]  (p1) -- (p2);
        \draw[blueedge] (p2) -- (p7);
        \draw[rededge]  (p7) -- (p8);
        \draw[blueedge] (p8) -- (p3);
        \draw[rededge]  (p3) -- (p4);
        \draw[blueedge] (p4) -- (p18);
        \draw[rededge]  (p18) -- (p17);

        \node[label, above=5pt of p1]  {$\pi(4) = 1$};
        \node[label, below=5pt of p2]  {$\pi(1) = 2$};
        \node[perm,  above=5pt of p7]  {$\pi(2) = 7$};
        \node[label, below=5pt of p8]  {$\pi(8) = 8$};
        \node[perm,  above=5pt of p3]  {$\pi(7) = 3$};
        \node[label, below=5pt of p4]  {$\pi(6) = 4$};
        \node[perm,  above=5pt of p18] {$\pi(5) = 18$};
        \node[label, below=5pt of p17] {$\pi(3) = 17$};

        \draw[<->, >=Stealth, line width=1pt, black!60]
            (2,1.15) to[out=65,in=115,looseness=1.25] (5.3,1.15);

        \end{tikzpicture}
    \end{minipage}
    \hspace*{-0.9cm}%
    \begin{minipage}[c]{0.34\textwidth}
        %\centering
        \begin{tikzpicture}[
            x=0.95cm,
            y=0.95cm,
            point/.style={circle, fill=black!75, inner sep=2.2pt},
            rededge/.style={red!65, line width=1.2pt},
            blueedge/.style={blue!65, line width=1.2pt},
            label/.style={font=\normalsize},
            perm/.style={font=\normalsize}
        ]

        % Regular hexagon vertices
        \coordinate (A) at (-2,0);
        \coordinate (B) at (-1,{sqrt(3)});
        \coordinate (C) at ( 1,{sqrt(3)});
        \coordinate (D) at ( 2,0);
        \coordinate (E) at ( 1,{-sqrt(3)});
        \coordinate (F) at (-1,{-sqrt(3)});

        % Dashed symmetry axis
        \draw[densely dashed, black!40] (0,2.3) -- (0,-2.3);

        % Vertices
        \node[point] (pA) at (A) {};
        \node[point] (pB) at (B) {};
        \node[point] (pC) at (C) {};
        \node[point] (pD) at (D) {};
        \node[point] (pE) at (E) {};
        \node[point] (pF) at (F) {};

        % Alternating coloured edges
        \draw[rededge]  (pA) -- (pB);
        \draw[blueedge] (pB) -- (pC);
        \draw[rededge]  (pC) -- (pD);
        \draw[blueedge] (pD) -- (pE);
        \draw[rededge]  (pE) -- (pF);
        \draw[blueedge] (pF) -- (pA);

        % Labels
        \node[label] at (-2.98,0.12) {$\pi(10)=2$};
        \node[label] at (-1.18,{sqrt(3)+0.28}) {$\pi(1)=1$};
        \node[label] at ( 2.93,0.05) {$\pi(7)=6$};

        \node[perm] at (1.72,{sqrt(3)+0.28}) {$\pi(2)=5$};
        \node[perm] at (-1.45,{-sqrt(3)-0.35}) {$\pi(9)=12$};
        \node[perm] at ( 1.35,{-sqrt(3)-0.35}) {$\pi(8)=11$};

        \end{tikzpicture}
    \end{minipage}
    \caption{Reflection of a path (left) and of a cycle (right) in Lemma~\ref{lm:eigeneven}.}
    \label{fig:path-cycle-reflection}
\end{figure}

\end{proof}

\begin{remark}
    An alternative proof of Lemma~\ref{lm:eigeneven} can be given using representation theory.
\end{remark}

\begin{proof}[Proof of Theorem~\ref{th:main}]

By Lemma~\ref{lm:key}(ii) all $a_i$ are nonnegative and so the largest eigenvalue is $\lambda_0$. By the Perron--Frobenius theorem $\lambda_0$ is simple.
By Lemma~\ref{lm:key}(iii) all eigenvalues except $\lambda_2 = \lambda_4 = \dots = \lambda_{2t}$ are positive; so the only negative (and hence it is the smallest) eigenvalue is $\lambda_2$.

Finally, by Lemma~\ref{lm:eigeneven} the characteristic vector of an atomic family belongs to the direct sum of the eigenspaces that correspond to the maximal and the minimal eigenvalues.
This means that the only inequality in~\eqref{eq:main} turns into equality and thus the bound~\eqref{eq:Hoffmanita} is tight.
\end{proof}

\section{Proof of Lemma~\ref{lm:key}} \label{sec:proof}

Put $\alpha := 0.1$ and $o := 7$ and recall that $n > \alpha^{-1}k^o$.

\subsection{Preliminary definitions and auxiliary lemmas}

Define
\[
G_{x,y,q} := \binom{x}{q} \binom{k-q}{y-q} \binom{n-k+q-x}{k-y+q-x}.
\]

\begin{lemma}\label{lm:sum_est}
Let $0\leq x,y\leq k$, and let $z=\min(x,y)$. Then
\[
\left|
\frac{P_x(y)}{(-1)^{x-z}G_{x,y,z}}-1
\right|
=
\left|
\frac{\sum_{q=0}^{x}(-1)^{x-q}G_{x,y,q}}
     {(-1)^{x-z}G_{x,y,z}}-1
\right|
\leq 4\alpha k^{3-o}.
\]
\end{lemma}

\begin{proof}
First, if $x<q$, $y<q$, or $x+y-k >q$, then the first, second, or
third binomial coefficient, respectively, in $G_{x,y,q}$ vanishes.
We now estimate the ratio of consecutive terms $G_{x,y,q}$ and
$G_{x,y,q+1}$ for
\[
\max(0,x+y-k) \leq q\leq z-1.
\]
We have
\[
\frac{G_{x,y,q}}{G_{x,y,q+1}}
 =
\frac{
\binom{x}{q}
\binom{k-q}{y-q}
\binom{n-k+q-x}{k-y+q-x}
}{
\binom{x}{q+1}
\binom{k-(q+1)}{y-(q+1)}
\binom{n-k+q+1-x}{k-y+q+1-x}
} =
\frac{\binom{x}{q}}{\binom{x}{q+1}}
\cdot
\frac{\binom{k-q}{y-q}}
     {\binom{k-(q+1)}{y-(q+1)}}
\cdot
\frac{\binom{n-k+q-x}{k-y+q-x}}
     {\binom{n-k+q+1-x}{k-y+q+1-x}}.
\]
Using
\[
\binom{N}{K}
=
\binom{N}{K+1}\frac{K+1}{N-K}
\qquad\text{and}\qquad
\binom{N}{K}
=
\binom{N+1}{K+1}\frac{K+1}{N+1},
\]
we obtain
\[
\frac{G_{x,y,q}}{G_{x,y,q+1}}
=
\frac{q+1}{x-q}
\cdot
\frac{k-q}{y-q}
\cdot
\frac{k-y+q+1-x}{n-k+q+1-x}
\leq
k\cdot k\cdot\frac{2k}{n}
\leq 
2\alpha k^{3-o}.
\]
It follows that
\[
\left|
\frac{\sum_{q=0}^{x}(-1)^{x-q}G_{x,y,q}}
     {(-1)^{x-z}G_{x,y,z}}-1
\right|
=
\left|
\frac{\sum_{q=\max(0,x+y-k)}^{z}(-1)^{x-q}G_{x,y,q}}
     {(-1)^{x-z}G_{x,y,z}}-1
\right| 
\leq
\]
\[
\sum_{r=1}^{z-\max(0,x+y-k)}
\left|
\frac{G_{x,y,z-r}}{G_{x,y,z}}
\right| 
\leq
\sum_{r=1}^{z-\max(0,x+y-k)}(2\alpha k^{3-o})^r 
\leq
\frac{2\alpha k^{3-o}}{1-2\alpha k^{3-o}}
\leq
4\alpha k^{3-o}.
\]
\end{proof}

Define
\[
w_{x,y}
=
\begin{cases}
\displaystyle
\binom{k-x}{y-x}\frac{1}{(k-y)!},
& x\leq y,\\[2mm]
\displaystyle
(-1)^{x-y}\binom{x}{y}\frac{1}{(k-x)!},
& x>y.
\end{cases}
\]

\begin{lemma}\label{lm:binom_est}
Let $0\leq x,y\leq 2t$, and let $z=\min(x,y)$. Then
\[
\left|
\frac{(-1)^{x-z}G_{x,y,z}}
     {w_{x,y}n^{k-\max(x,y)}}-1
\right|
\leq
2\alpha k^{2-o}.
\]
\end{lemma}

\begin{proof}
We consider the two cases $x\leq y$ and $x>y$.
\begin{itemize}
    \item Suppose that $x\leq y$. Then $x=z$, and
    \[
    G_{x,y,x}
    =
    \binom{k-x}{y-x}\binom{n-k}{k-y}.
    \]
    Hence
    \[
    \frac{(-1)^{x-z}G_{x,y,z}}
         {w_{x,y}n^{k-\max(x,y)}}
    =
    \frac{\prod_{i=0}^{k-y-1}(n-k-i)}{n^{k-y}}.
    \]
    This quantity is clearly at most one, while
    \[
    \frac{\prod_{i=0}^{k-y-1}(n-k-i)}{n^{k-y}}
    \geq
    \left(\frac{n-2k}{n}\right)^k
    \geq
    1-\frac{2k^2}{n}
    \geq
    1-2\alpha k^{2-o}.
    \]

    \item Suppose that $x>y$. Then $z=y$, and
    \[
    G_{x,y,y}
    =
    \binom{x}{y}\binom{n-k-x+y}{k-x}.
    \]
    Hence
    \[
    \frac{(-1)^{x-z}G_{x,y,z}}
         {w_{x,y}n^{k-\max(x,y)}}
    =
    \frac{\prod_{i=0}^{k-x-1}(n-k-x+y-i)}{n^{k-x}}.
    \]
    Again, this quantity is clearly at most one, since
    $n-k-x+y\leq n$, while
    \[
    \frac{\prod_{i=0}^{k-x-1}(n-k-x+y-i)}{n^{k-x}}
    \geq
    \left(
        \frac{n-k-x+y-(k-x)}{n}
    \right)^k 
    \geq
    \left(\frac{n-2k}{n}\right)^k
    \geq
    1-\frac{2k^2}{n}
    \geq
    1-2\alpha k^{2-o}.
    \]
\end{itemize}
\end{proof}

\begin{corollary}\label{col:est}
Combining Lemmas~\ref{lm:sum_est} and~\ref{lm:binom_est}, we obtain
\[
\left|
\frac{P_x(y)}{w_{x,y}n^{k-\max(x,y)}}-1
\right|
\leq
6\alpha k^{3-o}.
\]
\end{corollary}

\begin{lemma}\label{lm:ratio_est}
Let $0\leq x<k$ and $0\leq y\leq k$. Then
\[
\frac{1}{k}
\leq
\left|
\frac{w_{x+1,y}}{w_{x,y}}
\right|
\leq
\frac{k^2}{2}.
\]
\end{lemma}

\begin{proof}
We consider the three cases $x<y$, $x=y$, and $x>y$.
\begin{itemize}
    \item Suppose that $x<y$. Then
    \[
    w_{x,y}
    =
    \binom{k-x}{y-x}\frac{1}{(k-y)!},
    \qquad
    w_{x+1,y}
    =
    \binom{k-x-1}{y-x-1}\frac{1}{(k-y)!}.
    \]
    Hence
    \[
    \left|
    \frac{w_{x+1,y}}{w_{x,y}}
    \right|
    =
    \frac{y-x}{k-x},
    \]
    which is clearly at least $\frac{1}{k}$ and at most $1$.

    \item Suppose that $x=y$. Then
    \[
    w_{x,y}=w_{y,y}=\frac{1}{(k-y)!},
    \qquad
    w_{x+1,y}=w_{y+1,y}
    =
    -\frac{y+1}{(k-y-1)!}.
    \]
    Hence
    \[
    \left|
    \frac{w_{x+1,y}}{w_{x,y}}
    \right|
    =
    (k-y)(y+1),
    \]
    which is clearly at least $1$ and at most
    \[
    \frac{(k+1)^2}{4}<\frac{k^2}{2}.
    \]

    \item Suppose that $x>y$. Then
    \[
    w_{x,y}
    =
    (-1)^{x-y}\binom{x}{y}\frac{1}{(k-x)!},
    \qquad
    w_{x+1,y}
    =
    (-1)^{x-y+1}\binom{x+1}{y}\frac{1}{(k-x-1)!}.
    \]
    Hence
    \[
    \left|
    \frac{w_{x+1,y}}{w_{x,y}}
    \right|
    =
    \frac{(x+1)(k-x)}{x-y+1},
    \]
    which is also clearly at least $1$ and at most
    \[
    \frac{(k+1)^2}{4}<\frac{k^2}{2}.
    \]
\end{itemize}
\end{proof}

\begin{lemma}\label{lm:matr_det}
Let $A=(a_{i,j})_{i,j=1}^{N}$ be a matrix such that
\[
a_{i,j}=b_{i,j}X^{r_{i,j}}.
\]
Suppose that positive numbers $K,L$ and a permutation
$\pi\in S_N$ satisfy the following conditions:
\begin{itemize}
    \item $X>X_0:=4KLN$;

    \item for all $1\leq i\leq N-1$ and $1\leq j\leq N$ we have
    \[
    K^{-1} \leq \left| \frac{b_{i+1,j}}{b_{i,j}} \right| \leq L;
    \]

    \item for all $1\leq i,j\leq N$ we have
    \[
    r_{\pi(j),j}\geq r_{i,j};
    \]

    \item for every $j$ and every $i>\pi(j)$ we have
    \[
    r_{\pi(j),j}-r_{i,j} \geq i-\pi(j).
    \]
\end{itemize}
Then
\[
\frac{\det A}
     {\sign(\pi)\prod_{j=1}^{N}a_{\pi(j),j}}
\in
\left(
1-2X^{-1}KLN,
1+2X^{-1}KLN
\right).
\]
\end{lemma}

\begin{proof}
For a permutation $\rho\in S_N$, define its depth
\[
\ddp(\rho)
:=
\sum_{j:\, j <\rho(j)}
\bigl(\rho(j)-j\bigr),
\]
and put
\[
R_\ell( \pi) := \{\rho\in S_N:\ddp(\rho \circ \pi^{-1})=\ell\}.
\]
Obviously, there exists a natural bijection between $R_\ell(\pi)$ and $R_{\ell}(id)$.

Fix $\rho\in R_\ell(\pi)$. We estimate
\[
\left|
\prod_{j=1}^{N}a_{\rho(j),j}
\right|
\]
in terms of
\[
\left|
\prod_{j=1}^{N}a_{\pi(j),j}
\right|.
\]
If $\rho(j)>\pi(j)$, then
\[
\left|
\frac{b_{\rho(j),j}}{b_{\pi(j),j}}
\right|
=
\prod_{i=\pi(j)}^{\rho(j)-1}
\left|
\frac{b_{i+1,j}}{b_{i,j}}
\right|
\leq
L^{\rho(j)-\pi(j)}.
\]
If $\rho(j)<\pi(j)$, then
\[
\left|
\frac{b_{\rho(j),j}}{b_{\pi(j),j}}
\right|
=
\prod_{i=\rho(j)}^{\pi(j)-1}
\left|
\frac{b_{i,j}}{b_{i+1,j}}
\right|
\leq
K^{\pi(j)-\rho(j)}.
\]
Furthermore,
\[
\left|
\frac{\prod_{j=1}^{N}a_{\rho(j),j}}
     {\prod_{j=1}^{N}a_{\pi(j),j}}
\right|
=
\prod_{j=1}^{N}
\left|
\frac{b_{\rho(j),j}}{b_{\pi(j),j}}
\right|
X^{\sum_{j=1}^{N}(r_{\rho(j),j}-r_{\pi(j),j})} 
\leq
\prod_{j:\,\rho(j)>\pi(j)}
L^{\rho(j)-\pi(j)}
\prod_{j:\,\rho(j)<\pi(j)}
K^{\pi(j)-\rho(j)}
X^{-m} = K^mL^mX^{-m}.
\]
because 
\[
\sum_{j=1}^{N}(r_{\rho(j),j}-r_{\pi(j),j}) \leq \sum_{j: \pi(j)< \rho(j)}(\pi(j) - \rho(j)) = \sum_{i:i < \rho(\pi^{-1}(i))}(i - \rho(\pi^{-1}(i))) = - \ddp(\rho \circ \pi^{-1}) = -m.
\]

We now estimate $|R_\ell(\pi)| = |R_{\ell}(id)|$.

\begin{lemma}
For every $\ell\geq1$, one has
\[
|R_\ell (id)|
\leq
2^{\ell -1}\binom{N+\ell -2}{\ell }.
\]
\end{lemma}

\begin{proof}
The following argument is essentially contained in~\cite{petersen2015depth}.
We identify a permutation matrix with the corresponding permutation
$\rho\in S_N$. 

We first construct a convenient transposition factorization of
$\rho$.  Write $t_{a,b}$ for the transposition interchanging $a$
and $b$, where $a<b$.  We claim that every $\rho\in S_N$ admits a
factorization
\begin{equation}
    \rho
    =
    t_{a_1,b_1}\cdots t_{a_r,b_r}
    \label{eq:shallow-factorization}
\end{equation}
with the following properties:
\begin{enumerate}
    \item the upper endpoints $b_1,\ldots,b_r$ are pairwise distinct;
    \item the sequence $b_1,\ldots,b_r$ is first strictly decreasing
          and then strictly increasing;
    \item
    \[
        \ddp(\rho)
        =
        \sum_{j=1}^{r}(b_j-a_j).
    \]
\end{enumerate}

We prove this claim by induction on $N$.  If $\rho(N)=N$, we simply
apply the induction hypothesis to the restriction of $\rho$ to
$[N-1]$.

Suppose now that $\rho(N)\neq N$, and put
$p:=\rho^{-1}(N)$, $q:=\rho(N)$.
Thus $p,q<N$.  Define
\[
    \sigma := \rho\,t_{p,N} = t_{q,N}\rho.
\]
Then $\sigma(N)=N$, and the only changed values are
$\rho(p)=N$, $\rho(N)=q$, which are replaced by $\sigma(p)=q$, $\sigma(N)=N$.
Consequently,
\begin{equation} \label{eq:depth-reduction}
    \ddp(\rho)-\ddp(\sigma) =
    (N-p)-\max\{q-p,0\} = N-\max\{p,q\}.
\end{equation}

If $p\geq q$, then $\rho=\sigma t_{p,N}$, and the new transposition has contribution $N-p$.  If $q>p$, then $\rho=t_{q,N}\sigma$,
and the new transposition has contribution $N-q$.  In either case,
the contribution is exactly the quantity in~\eqref{eq:depth-reduction}.

The new upper endpoint is $N$, while all upper endpoints in the
factorization of $\sigma$ are at most $N-1$.  In the first case it
is appended to the factorization, and in the second case it is
prepended.  Hence the upper endpoints remain pairwise distinct and
form a sequence which is first decreasing and then increasing.
This proves the claim.

We now count the possible factorizations
\eqref{eq:shallow-factorization} for permutations in $R_\ell $.  Suppose
that the factorization contains $r$ transpositions.  Since every
$b_j-a_j$ is positive, $1\leq r\leq \ell $,
and the distinct upper endpoints form an $r$-element subset of
$\{2,\ldots,N\}$.  Hence they can be chosen in $\binom{N-1}{r}$
ways.

For a fixed set of upper endpoints, there are at most $2^{r-1}$ orders which are first decreasing and then increasing.  Indeed, the
smallest upper endpoint is the unique minimum of the sequence, and
each of the other $r-1$ endpoints may be placed either to its left
or to its right; the order on each side is then forced.

Finally, put $\ell_j:=b_j-a_j$.
The numbers $\ell_1,\ldots,\ell_r$ are positive integers satisfying
\[
    \ell_1+\cdots+\ell_r=\ell .
\]
The number of such compositions is $\binom{\ell -1}{r-1}$.

Once $b_j$ and $\ell_j$ are fixed, the lower endpoint is determined
by $a_j=b_j-\ell_j$.
Some choices may fail to satisfy $a_j \geq 1$, but ignoring this
restriction only enlarges the count.  Since an ordered product of
transpositions determines the permutation, we obtain
\[
    |R_\ell (id)|
    \leq
    \sum_{r=1}^{\min\{\ell ,N-1\}}
    2^{r-1}
    \binom{N-1}{r}
    \binom{\ell -1}{r-1}.
\]

Using $2^{r-1}\leq2^{\ell -1}$ and Vandermonde's identity, we conclude
that
\[
    |R_\ell (id)| \leq
    2^{\ell -1}
    \sum_{r\geq1}
    \binom{N-1}{r}
    \binom{\ell -1}{r-1} =
    2^{\ell -1}
    \sum_{r\geq1}
    \binom{N-1}{r}
    \binom{\ell -1}{\ell -r} =
    2^{\ell -1}\binom{N+\ell -2}{\ell }.
\]
\end{proof}

Observe that
\[
\begin{aligned}
|R_\ell (\pi)|= |R_{\ell }(id)| \leq 2^{\ell -1}\binom {N+\ell -2}{\ell }
&=
\frac{2^{\ell -1}\prod_{i=0}^{\ell -1}(N+\ell -2-i)}{\ell !} \\
&=
(N-1)
\prod_{i=1}^{\ell -1}
\left(
\frac{2(N-1+i)}{i+1}
\right) \\
&\leq
(N-1)N^{\ell -1}
<
N^\ell .
\end{aligned}
\]

We can now estimate the determinant expansion:
\[
\left|
\frac{\det A}
     {\sign(\pi)\prod_{j=1}^{N}a_{\pi(j),j}}
-1
\right|
=
\left|
\frac{
\sum_{\rho\in S_N}
\sign(\rho)
\prod_{j=1}^{N}a_{\rho(j),j}
}
{
\sign(\pi)
\prod_{j=1}^{N}a_{\pi(j),j}
}
-1
\right| \leq
\sum_{\ell =1}^{\infty}
\sum_{\rho\in R_\ell (\pi)}
\left|
\frac{\prod_{j=1}^{N}a_{\rho(j),j}}
     {\prod_{j=1}^{N}a_{\pi(j),j}}
\right| \leq
\]
\[
\sum_{\ell =1}^{\infty}
|R_\ell (\pi)|K^\ell L^\ell X^{-\ell } \leq
\sum_{\ell =1}^{\infty}
(X^{-1}KLN)^\ell  =
\frac{X^{-1}KLN}{1-X^{-1}KLN} \leq
2X^{-1}KLN.
\]
\end{proof}

\begin{corollary}\label{col:full_rk}
If the assumptions of Lemma~\ref{lm:matr_det} are satisfied, then
$A$ is nonsingular. Indeed, its determinant is nonzero, since
\[
0
\notin
\left(
1-2X^{-1}KLN,
1+2X^{-1}KLN
\right).
\]
\end{corollary}

\subsection{Proof of the key lemma} 

Define the matrix
\[
D=(d_{u,v})_{u,v=1}^{t},
\qquad
d_{u,v}=P_{2u}(2v-1).
\]
The equal-eigenvalue equations are
\[
\sum_{v=1}^{t}d_{u,v}a_{2v-1}=-1,
\qquad
1\leq u\leq t.
\]
By Cramer's rule,
\[
a_{2v-1}
=
\frac{\det D_v}{\det D},
\]
where $D_v$ is obtained from $D$ by replacing its $v$th column with
the column vector whose entries are all equal to $-1$.

By Corollary~\ref{col:est},
\[
\left|
\frac{d_{u,v}}
     {w_{2u,2v-1}n^{k-\max(2u,2v-1)}}-1
\right|
\leq
6\alpha k^{3-o}.
\]
Set
\[
\Delta_{u,v}
:=
\frac{d_{u,v}}
     {w_{2u,2v-1}n^{k-\max(2u,2v-1)}}-1.
\]
Also put
\[
e_{u,v}
:=
d_{u,v}n^{\max(2u,2v-1)-k}.
\]
Then
\[
\frac{e_{u,v}}{w_{2u,2v-1}}
=
1+\Delta_{u,v}.
\]
Define
\[
K_0
:=
\frac{k^2(1+6\alpha k^{3-o})}{1-6\alpha k^{3-o}},
\qquad
L_0
:=
\frac{(1+6\alpha k^{3-o})k^4}
     {4(1-6\alpha k^{3-o})}.
\]

\begin{proposition}\label{pr:est1}
Let $1\leq u\leq t-1$ and $1\leq v\leq t$. Then
\[
K_0^{-1}
=
\frac{1-6\alpha k^{3-o}}{k^2(1+6\alpha k^{3-o})}
\leq
\left|
\frac{e_{u+1,v}}{e_{u,v}}
\right|
\leq
\frac{k^4(1+6\alpha k^{3-o})}
     {4(1-6\alpha k^{3-o})}
=
L_0.
\]
\end{proposition}

\begin{proof}
We have
\[
\left|
\frac{e_{u+1,v}}{e_{u,v}}
\right|
=
\left|
\frac{e_{u+1,v}}{w_{2u+2,2v-1}}
\right|
\left|
\frac{w_{2u+2,2v-1}}{w_{2u,2v-1}}
\right|
\left|
\frac{w_{2u,2v-1}}{e_{u,v}}
\right|.
\]
Consequently,
\[
\frac{1-6\alpha k^{3-o}}{k^2(1+6\alpha k^{3-o})}
\leq
\left|
\frac{e_{u+1,v}}{w_{2u+2,2v-1}}
\cdot
\frac{w_{2u+2,2v-1}}{w_{2u,2v-1}}
\cdot
\frac{w_{2u,2v-1}}{e_{u,v}}
\right|
\leq
\frac{k^4(1+6\alpha k^{3-o})}
     {4(1-6\alpha k^{3-o})}.
\]
\end{proof}

It is also immediate that
\begin{equation}\label{ineq:n_big}
4K_0L_0t = 4t\frac{k^6(1+6\alpha k^{3-o})^2}
        {4(1-6\alpha k^{3-o})^2}
< k^7 < \alpha ^{-1}k^o \leq n.
\end{equation}

We first analyze $D$. Consider the matrix
\[
H=(h_{u,v})_{u,v=1}^{t},
\qquad
h_{u,v}
=
d_{u,v}n^u
=
e_{u,v}n^{k+u-\max(2u,2v-1)}.
\]
Clearly,
\[
\det H
=
\det D\cdot n^{t(t+1)/2}.
\]
We verify that $H$ satisfies the assumptions of
Lemma~\ref{lm:matr_det} with
\[
K=K_0,
\qquad
L=L_0,
\qquad
\pi=\operatorname{id}.
\]
The first two assumptions follow from~\eqref{ineq:n_big} and
Proposition~\ref{pr:est1}, respectively. It remains to verify that
\[
r_{u,v}
:=
k+u-\max(2v-1,2u)
\]
satisfies the exponent conditions of Lemma~\ref{lm:matr_det}.

Indeed,
\[
r_{v,v}=k-v.
\]
If $u<v$, then
\[
r_{u,v}
\leq
k+u-2v+1
\leq
k+(v-1)-2v+1
=
k-v.
\]
If $u\geq v$, then
\[
r_{u,v}
\leq
k+u-2u
=
k-u
\leq
k-v.
\]
This proves the first exponent condition.

We also have to prove that, whenever $u>v$,
\[
r_{v,v}-r_{u,v}\geq u-v.
\]
For $u>v$,
\[
r_{u,v}=k+u-2u=k-u,
\]
and therefore
\[
r_{v,v}-r_{u,v}
=
(k-v)-(k-u)
=
u-v.
\]
Thus the second exponent condition is also satisfied.

\begin{figure}[h]
    \centering
    \begin{tikzpicture}[scale=1]
    \def\eps{0.08}
    % Grid
    \draw[thick] (0,0) grid (5,5);

    \foreach \n in {1,7,13,19,25}{
    \pgfmathtruncatemacro{\r}{(\n-1)/5}
    \pgfmathtruncatemacro{\c}{mod(\n-1,5)}
    \fill[brown!60]
        (\c+\eps,4-\r+\eps)
        rectangle
        (\c+1-\eps,5-\r-\eps);
        }
        
    % Numbers
    \foreach \row in {0,...,4}{
        \foreach \col in {0,...,4}{
            \pgfmathtruncatemacro{\num}{9+\row- max(2*\row, 2*\col-1)}
            \node at (\col+0.5,4.5-\row) {\num};
        }
    }
    \end{tikzpicture}
    \caption{$r_{u,v}=k+u-\max(2v-1,2u)$ for $t=5$.}
    \label{fig:table1}
\end{figure}

Applying Lemma~\ref{lm:matr_det}, we obtain
\[
\frac{\det H}{\prod_{v=1}^{t}h_{v,v}}
\in
\left(
1-2n^{-1}K_0L_0t,
1+2n^{-1}K_0L_0t
\right)
\subset
\left(
1-\alpha k^{7-o},
1+\alpha k^{7-o}
\right).
\]
On the other hand,
\[
\frac{\det H}{\prod_{v=1}^{t}h_{v,v}}
=
\frac{
\det D\cdot n^{t(t+1)/2}
}{
\left(\prod_{v=1}^{t}e_{v,v}\right)
 n^{\sum_{v=1}^{t}(k+v-2v)}
}.
\]
It follows that
\[
\frac{\det D}
     {\left(\prod_{v=1}^{t}e_{v,v}\right)n^{kt-t(t+1)}}
=
\frac{\det D}
     {\left(\prod_{v=1}^{t}e_{v,v}\right)n^{t(t-1)}}
=
1+\varepsilon,
\]
where
\[
|\varepsilon|<\alpha k^{7-o}.
\]
In particular, $D$ is nonsingular. Hence the system
\[
\lambda_{2r}(a_1,a_3,\ldots,a_{2t-1})=-1,
\qquad
1\leq r\leq t,
\]
has a unique solution.

We now analyze $D_s$. Consider the matrix
\[
B_s=(b_{s,u,v})_{u,v=1}^{t},
\qquad
b_{s,u,v}
=
(D_s)_{u,v}n^{\min(u,s)}
=
c_{s,u,v}n^{r_{s,u,v}},
\]
where
\[
c_{s,u,v}
=
\begin{cases}
    e_{u,v}, & v\neq s,\\
    -1, & v=s,
\end{cases}
\qquad
r_{s,u,v}
=
\begin{cases}
    k+\min(u,s)-\max(2u,2v-1), & v\neq s,\\
    \min(u,s), & v=s.
\end{cases}
\]
Clearly,
\[
\det B_s
=
\det D_s\cdot n^{\sum_{u=1}^{t}\min(u,s)}.
\]
We show that $B_s$ satisfies the assumptions of
Lemma~\ref{lm:matr_det} with
\[
K=K_0,
\qquad
L=L_0,
\qquad
\pi_s(v)
=
\begin{cases}
    v, & v<s,\\
    t, & v=s,\\
    v-1, & v>s.
\end{cases}
\]
The first assumption follows from~\eqref{ineq:n_big}. For $v\neq s$,
the second assumption follows from Proposition~\ref{pr:est1}; for
$v=s$, we have
\[
K^{-1}
<
1
=
\left|
\frac{c_{s,u+1,s}}{c_{s,u,s}}
\right|
<
L.
\]
It remains to check the exponent conditions for $r_{s,u,v}$. We
consider separately the columns with $v<s$, $v=s$, and $v>s$.

Suppose first that $v<s$. We must verify that
\[
r_{s,v,v}\geq r_{s,u,v}.
\]
We have
\[
r_{s,v,v}
=
k+\min(v,s)-\max(2v,2v-1)
=
k-v.
\]
If $u<v$, then
\[
\begin{aligned}
r_{s,u,v}
&=
k+\min(u,s)-\max(2u,2v-1) \\
&=
k+u-(2v-1) \\
&\leq
k+v-1-2v+1
=
k-v.
\end{aligned}
\]
If $u\geq v$, then
\[
r_{s,u,v}
\leq
k+u-2u
=
k-u
\leq
k-v.
\]
This proves the first condition. If
$u>\pi_s(v)=v$, then
\[
r_{s,v,v}-r_{s,u,v}
\geq
(k-v)-(k-u)
=
u-v,
\]
which proves the second condition.

Suppose next that $v=s$. The inequality
\[
r_{s,t,s}
=
\min(t,s)
=
s
\geq
\min(u,s)
=
r_{s,u,s}
\]
is immediate. Moreover, the inequality
$u>\pi_s(s)=t$ never occurs, so the second condition is vacuous.

Finally, suppose that $v>s$. The first condition is
\[
r_{s,v-1,v}\geq r_{s,u,v}.
\]
We have
\[
r_{s,v-1,v}
=
k+\min(s,v-1)-\max(2(v-1),2v-1)
=
k+s-(2v-1).
\]
For every $u$,
\[
r_{s,u,v} = k+\min(u,s)-\max(2u,2v-1) \leq k+s-(2v-1) = r_{s,v-1,v}.
\]
Thus the first condition is satisfied. For the second condition, let
$u>\pi_s(v)=v-1$. Since then $u\geq v$, we have
\[
r_{s,u,v}\leq k+s-2u,
\]
and hence
\[
r_{s,v-1,v}-r_{s,u,v} \geq 2u-(2v-1) \geq u-(v-1).
\]
Thus the second condition is satisfied in this case as well.

\begin{figure}[h]
    \centering
    \begin{tikzpicture}[scale=1]
    \def\eps{0.08}
    % Grid
    \draw[thick] (0,0) grid (5,5);

    \foreach \n in {1,7,14,20,23}{
    \pgfmathtruncatemacro{\r}{(\n-1)/5}
    \pgfmathtruncatemacro{\c}{mod(\n-1,5)}
    \fill[brown!60]
        (\c+\eps,4-\r+\eps)
        rectangle
        (\c+1-\eps,5-\r-\eps);
        }
        
    % Numbers
    \foreach \row in {0,...,4}{
        \foreach \col in {0,1, 3,4}{
             \pgfmathtruncatemacro{\num}{9+min(\row,2)- max(2*\row, 2*\col-1)}
            \node at (\col+0.5,4.5-\row) {\num};}
            
        \foreach \col in {2}{
             \pgfmathtruncatemacro{\num}{1+min(\row,2)}
            \node at (\col+0.5,4.5-\row) {\num};
        }
    }
\end{tikzpicture}
    \caption{$r_{s,u,v}$ for $t=5$ and $s=3$.}
    \label{fig:table2}
\end{figure}

We may therefore apply Lemma~\ref{lm:matr_det} and obtain
\[
\frac{\det B_s}
     {\sign(\pi_s)
      \prod_{v=1}^{t}b_{s,\pi_s(v),v}}
\in
\left(
1-2n^{-1}K_0L_0t,
1+2n^{-1}K_0L_0t
\right)
\subset
\left(
1-\alpha k^{7-o},
1+\alpha k^{7-o}
\right).
\]
Thus
\[
\frac{\det B_s}
     {\sign(\pi_s)
      \prod_{v=1}^{t}b_{s,\pi_s(v),v}}
=
\frac{
\det D_s\,n^{\sum_{u=1}^{t}\min(u,s)}
}{
\sign(\pi_s)
\left(\prod_{v=1}^{t}c_{s,\pi_s(v),v}\right)
 n^{\sum_{v=1}^{t}r_{s,\pi_s(v),v}}
}
=:
1+\varepsilon_s,
\]
where
\[
|\varepsilon_s|<\alpha k^{7-o}.
\]
Clearly,
\[
\sign(\pi_s)=(-1)^{t-s},
\]
since $\pi_s$ has exactly $t-s$ inversions.

We now compute
\[
\sum_{v=1}^{t}
\left(
    r_{s,\pi_s(v),v}-\min(s,v)
\right) =
\sum_{v=1}^{t}
\left(
    r_{s,\pi_s(v),v}-\min(s,\pi_s(v))
\right) =
\]
\[
\sum_{v=1}^{s-1}(k-2v)
+
\sum_{v=s+1}^{t}(k-2v+1) = 
k(t-1)-(s-1)s-t^2+s^2 =
t(t-2)+s.
\]
Consequently,
\[
\det D_s
=
(-1)^{t-s}(1+\varepsilon_s)
\left(
\prod_{v=1}^{t}c_{s,\pi_s(v),v}
\right)
n^{t(t-2)+s}.
\]
We can now write
\[
\begin{aligned}
a_{2s-1} &=
\frac{\det D_s}{\det D} =
(-1)^{t-s}
\frac{1+\varepsilon_s}{1+\varepsilon}
\cdot
\frac{\prod_{v=1}^{t}c_{s,\pi_s(v),v}}
     {\prod_{v=1}^{t}e_{v,v}}
 n^{s-t} \\
&=
(-1)^{t-s}
\left(
\frac{1+\varepsilon_s}{1+\varepsilon}
\right)
\frac{-\prod_{v=s+1}^{t}e_{v-1,v}}
     {\prod_{v=s}^{t}e_{v,v}}
 n^{s-t} \\
&=
(-1)^{t-s+1}
\left(
\frac{1+\varepsilon_s}{1+\varepsilon}
\right)
\frac{
\prod_{v=s+1}^{t}
\left(
    w_{2(v-1),2v-1}(1+\Delta_{v-1,v})
\right)
}{
\prod_{v=s}^{t}
\left(
    w_{2v,2v-1}(1+\Delta_{v,v})
\right)
}
n^{s-t} \\
&=
(-1)^{t-s+1}
\left(
\frac{1+\varepsilon_s}{1+\varepsilon}
\right)
\frac{
\prod_{v=s+1}^{t}(1+\Delta_{v-1,v})
}{
\prod_{v=s}^{t}(1+\Delta_{v,v})
}
\frac{
\prod_{v=s+1}^{t}w_{2v-2,2v-1}
}{
\prod_{v=s}^{t}w_{2v,2v-1}
}
n^{s-t}.
\end{aligned}
\]
Recall that
\[
w_{x,y}
=
\begin{cases}
\displaystyle
\binom{k-x}{y-x}\frac{1}{(k-y)!},
&x\leq y,\\[2mm]
\displaystyle
(-1)^{x-y}\binom{x}{y}\frac{1}{(k-x)!},
&x>y.
\end{cases}
\]
Therefore
\[
w_{2v,2v-1}
=
-\frac{2v}{(k-2v)!},
\qquad
w_{2v-2,2v-1}
=
\frac{k-2v+2}{(k-2v+1)!}.
\]

Observe that
\[
1-4\alpha k^{7-o} <
(1-\alpha k^{7-o})^2 <
\frac{1-\alpha k^{7-o}}{1+\alpha k^{7-o}} <
\frac{1+\varepsilon_s}{1+\varepsilon} <
\]
\[
\frac{1+\alpha k^{7-o}}{1-\alpha k^{7-o}}  <
\frac{1}{(1-\alpha k^{7-o})^2} <
1+4\alpha k^{7-o},
\]
and
\[
1-12\alpha k^{4-o} <
(1-6\alpha k^{3-o})^k <
\frac{(1-6\alpha k^{3-o})^t}
     {(1+6\alpha k^{3-o})^t}  <
\frac{\prod_{v=s+1}^{t}(1+\Delta_{v-1,v})}{\prod_{v=s}^{t}(1+\Delta_{v,v})} <
\]
\[
\frac{(1+6\alpha k^{3-o})^t}
     {(1-6\alpha k^{3-o})^t} <
\frac{1}{(1-6\alpha k^{3-o})^k} <
\frac{1}{1-6\alpha k^{4-o}} <
1+12\alpha k^{4-o}.
\]
Here we repeatedly used Bernoulli's inequality and
\[
\frac{1}{1-x}<1+2x,
\qquad
0<x<\frac12.
\]
It follows that
\[
1-5\alpha k^{7-o} <
(1-4\alpha k^{7-o})(1-12\alpha k^{4-o}) <
\left(
\frac{1+\varepsilon_s}{1+\varepsilon}
\right) \frac{\prod_{v=s+1}^{t}(1+\Delta_{v-1,v})}{\prod_{v=s}^{t}(1+\Delta_{v,v})} <
\]
\[
(1+12\alpha k^{4-o})(1+4\alpha k^{7-o})  <
1+5\alpha k^{7-o}.
\]
Define $\tau_s$ by
\[
\left(
\frac{1+\varepsilon_s}{1+\varepsilon}
\right)
\frac{
\prod_{v=s+1}^{t}(1+\Delta_{v-1,v})
}{
\prod_{v=s}^{t}(1+\Delta_{v,v})
}
=
1+\tau_s.
\]
Then
\[
|\tau_s|<5\alpha k^{7-o}.
\]
Consequently,
\[
\begin{aligned}
a_{2s-1}
&=
(-1)^{t-s+1}(1+\tau_s)
\frac{\prod_{v=s+1}^{t}w_{2v-2,2v-1}}
     {\prod_{v=s}^{t}w_{2v,2v-1}}
 n^{s-t} \\
&=
(-1)^{t-s+1}(1+\tau_s)
\prod_{v=s+1}^{t}
\left(
\frac{k-2v+2}{(k-2v+1)!}
\right)
\prod_{v=s}^{t}
\left(
-\frac{(k-2v)!}{2v}
\right)
n^{s-t} \\
&=
(1+\tau_s)
\prod_{v=s+1}^{t}
\left(
\frac{k-2v+2}{(k-2v+1)!}
\right)
\prod_{v=s}^{t}
\left(
\frac{(k-2v)!}{2v}
\right)
n^{s-t} \\
&=
(1+\tau_s)
\prod_{v=s+1}^{t}
\left(
\frac{(k-2v+2)(k-2v)!}
     {(k-2v+1)!\,2v}
\right)
\frac{(k-2s)!}{2s}
n^{s-t} \\
&=
(1+\tau_s)
\prod_{v=s+1}^{t}
\left(
\frac{t-v+1}{(k-2v+1)v}
\right)
\frac{(k-2s)!}{2s}
n^{s-t} \\
&=
(1+\tau_s)
\frac{(t-s)!(k-2s)!(s-1)!}
     {2(2t-2s-1)!!\,t!}
 n^{s-t}.
\end{aligned}
\]
In particular,
\[
a_{2s-1}>0
\qquad
\text{for every }1\leq s\leq t.
\]

It remains to prove that all the odd-indexed eigenvalues
$\lambda_{2j-1}$ are positive for $1\leq j\leq t$. Recall that
\[
\lambda_{2j-1}
=
\sum_{s=1}^{t}a_{2s-1}P_{2j-1}(2s-1).
\]
We have already proved that every $a_{2s-1}$ is positive. By
Corollary~\ref{col:est}, the quantity $P_{2j-1}(2s-1)$ has the same
sign as $w_{2j-1,2s-1}$. The latter is positive: if
$2j-1\leq2s-1$, this follows from the first line in the definition
of $w_{x,y}$, whereas if $2j-1>2s-1$, then the sign factor is
\[
(-1)^{(2j-1)-(2s-1)}=1.
\]
Therefore every summand in the expression for $\lambda_{2j-1}$ is
positive, and hence
\[
\lambda_{2j-1}>0
\qquad
\text{for every }1\leq j\leq t.
\]
This completes the proof of Lemma~\ref{lm:key}.

\section{Uniqueness} \label{sec:Uniqueness}

\subsection{The eigenmatrices of the Johnson scheme} \label{subsec:basicsofLP}

Assume that $n\geq 2k$, and recall the notation
\[
    \mathcal X=\binom{[n]}{k}, \qquad
    (A_j)_{F,F'}
    =
    \mathbf 1_{\{|F\cap F'|=j\}}
\]
be the standard basis of the Bose--Mesner algebra of the Johnson scheme~\cite[Chapter~6]{GodsilMeagher2015Johnson}.  Let $E_0,\ldots,E_k$
be its primitive idempotents, and let $V_e=\operatorname{im}E_e$.
Thus
\[
    \mathbb R^\mathcal X=V_0\oplus\cdots\oplus V_k,
    \qquad
    E_0=\frac{1}{|\mathcal{X}|}J.
\]

In the convention used here, the first eigenmatrix $P$ with entries $P_e(j)$, $0\leq e,j \leq k$
is defined by
\[
    A_j=\sum_{e=0}^{k}P_e(j)E_e.
\]
Equivalently, $P_e(j)$ is the eigenvalue of $A_j$ on $V_e$.  These eigenvalues, also known as the dual Hahn or Eberlein
polynomials, are given by~\eqref{eq:Hahn}. 

The second eigenmatrix
\[
    Q=(Q_e(j))_{e,j=0}^{k}
\]
is defined by the inverse expansion
\[
    E_e=\frac{1}{|\mathcal{X}|}\sum_{j=0}^{k}Q_e(j)A_j.
\]
Let
\[
    v_j
    :=
    \binom{k}{j}\binom{n-k}{k-j}
\]
be the valency of $A_j$, and let
\[
    \mu_e
    :=
    \dim V_e
    =
    \binom{n}{e}-\binom{n}{e-1},
\]
where $\binom{n}{-1}=0$.  Then $P_0(j)=v_j$, $Q_e(k)=\mu_e$, and the two eigenmatrices are related by
\begin{equation} \label{eq:QsymP}
    Q_e(j)=\frac{\mu_e}{v_j}P_e(j).
\end{equation}    

The matrices $P$ and $Q$ also satisfy the orthogonality relations
\[
    QP^{\mathsf T}=P^{\mathsf T}Q=|\mathcal{X}|\cdot I.
\]

If $\mathcal F\subseteq \mathcal X$, $\chi_{\mathcal F}$ is its characteristic
vector, and
\[
    d_j
    :=
    \frac{1}{|\mathcal F|}
    \chi_{\mathcal F}^{\mathsf T}A_j\chi_{\mathcal F},
\]
where we index the distance distribution by the intersection size.
Its dual inner distribution is
\[
    d_e^{\perp}
    :=
    \sum_{j=0}^{k}Q_e(j)d_j
    =
    \frac{|\mathcal{X}|}{|\mathcal F|}
    \left\|E_e\chi_{\mathcal F}\right\|^2.
\]
In particular,
\[
    d_e^{\perp}\geq 0,
    \qquad
    d_e^{\perp}=0
    \quad\Longleftrightarrow\quad
    E_e\chi_{\mathcal F}=0.
\]

\subsection{The distance distribution in an extremal eventown} \label{subsec:distdist}
 
Suppose that $n > 10k^7$. We have proved that every $k$-uniform eventown has size at most $\binom{m}{t}$.
Consider an extremal eventown $\mathcal E$ and let $\chi_\mathcal{E}$ be its characteristic vector.
The extremality of $\mathcal E$ implies that for $\chi_\mathcal{E}$ inequality~\eqref{eq:main} turns to equality.
Hence
\[
\chi_\mathcal{E} \in V_0\oplus V_2\oplus\cdots\oplus V_{2t}.
\]

This gives a system of equations $d_{2e+1}^{\perp}=0$ for $e = 0,\dots,t-1$. 
Also from the eventown condition we have $d_{2e+1} = 0$ for $e = 0,\dots,t-1$. 
Thus all $d_i$ can be determined if the following matrix has full rank
\begin{equation}
    \mathcal R_{n,k}
    :=
    \begin{pmatrix}
        Q_1(2) & Q_1(4) & \cdots & Q_1(2t)\\
        Q_3(2) & Q_3(4) & \cdots & Q_3(2t)\\
        \vdots & \vdots & & \vdots\\
        Q_{2t-1}(2) & Q_{2t-1}(4) & \cdots & Q_{2t-1}(2t)
    \end{pmatrix}.
    \label{eq:distance-rank-matrix}
\end{equation}

By~\eqref{eq:QsymP} it suffices to prove that the corresponding minor of $P$ is nonsingular; this is the content of the following lemma.

\begin{lemma} \label{lm:fullrank}
The matrix
\[
F=(f_{u,v})_{u,v=1}^{t},
\qquad
f_{u,v}=P_{2u-1}(2v),
\]
has full rank whenever
\[
n\geq \alpha ^{-1}k^{o}.
\]
\end{lemma}

\begin{proof}

Set
\[
g_{u,v} := f_{u,v}n^{\max(2u-1,2v)-k}.
\]

By Corollary~\ref{col:est}, we have
\[
\left| \frac{g_{u,v}}{w_{2u-1,2v}}-1 \right| \leq 6\alpha k^{3-o}.
\]

\begin{proposition}\label{pr:est2}
Let $1\leq u\leq t-1$ and $1\leq v\leq t$. Then
\[
K_0^{-1}
=
\frac{1-6\alpha k^{3-o}}{k^2(1+6\alpha k^{3-o})}
\leq
\left|
\frac{g_{u+1,v}}{g_{u,v}}
\right|
\leq
\frac{k^4(1+6\alpha k^{3-o})}
     {4(1-6\alpha k^{3-o})}
=
L_0.
\]
\end{proposition}

\begin{proof}
We have
\[
\left|
\frac{g_{u+1,v}}{g_{u,v}}
\right|
=
\left|
\frac{g_{u+1,v}}{w_{2u+1,2v}}
\right|
\left|
\frac{w_{2u+1,2v}}{w_{2u-1,2v}}
\right|
\left|
\frac{w_{2u-1,2v}}{g_{u,v}}
\right|.
\]
Consequently,
\[
\frac{1-6\alpha k^{3-o}}{k^2(1+6\alpha k^{3-o})}
\leq
\left|
\frac{g_{u+1,v}}{w_{2u+1,2v}}
\right|
\left|
\frac{w_{2u+1,2v}}{w_{2u-1,2v}}
\right|
\left|
\frac{w_{2u-1,2v}}{g_{u,v}}
\right|
\leq
\frac{k^4(1+6\alpha k^{3-o})}
     {4(1-6\alpha k^{3-o})}.
\]
\end{proof}

We show that the matrix
\[
f_{u,v}
=
g_{u,v}n^{k-\max(2u-1,2v)}
\]
satisfies the assumptions of Lemma~\ref{lm:matr_det} with
\[
K=K_0,
\qquad
L=L_0,
\qquad
\pi=\operatorname{id}.
\]

The first two assumptions follow from
\eqref{ineq:n_big} and Proposition~\ref{pr:est2}, respectively.
It remains to verify that
\[
r_{u,v}
:=
k-\max(2u-1,2v)
\]
together with the identity permutation satisfies the last two
assumptions of Lemma~\ref{lm:matr_det}.

Indeed,
\[
r_{u,v} = k-\max(2u-1,2v) \leq k-2v = r_{v,v}.
\]
Moreover, if $u>v$, then
\[
r_{v,v}-r_{u,v} = (k-2v)-(k-2u+1) = 2(u-v)-1 \geq  u-v.
\]
Thus all the assumptions of Lemma~\ref{lm:matr_det} are satisfied.

We may therefore apply Lemma~\ref{lm:matr_det}, and
Corollary~\ref{col:full_rk} implies that the matrix $F$ is nonsingular.
\end{proof}

Since $\mathcal R_{n,k}$ is non-degenerate and $\sum d_i = |\mathcal{E}|$ all values of $d_i$ are determined provided that $n > 10k^7$.
Thus $d_i$ must coincide with the distribution of an atomic construction $\mathcal{A}_t$. So
\[
d_{2i} = \binom{t}{i} \binom{m-t}{t-i}
\]
for every $0 \leq i \leq t$.

\subsection{Combinatorial argument} \label{subsec:uniqueness}

Let $\mathcal E\subseteq\binom{[n]}{k}$ be an eventown of size $\binom{m}{t}$.
By the previous section we have
\[
d_{k-2} = t(m-t),
\]
so there exists a set $e$ from $\mathcal{E}$ which intersects at least $t(m-t)$ sets from $\mathcal{E}$ in exactly $k-2$ elements.

Let \[
\mathcal H_e
:=
\left\{
    A\cup B:
    A\in\binom e2,\ 
    B\in\binom{[n]\setminus e}{2}
\right\}
\]
be the set of all quadruples intersecting $e$ in exactly 2 elements. 
Choose distinct sets $c_1,\dots, c_{t(m-t)}$ intersecting $e$ in $k-2$ elements.
Let $C_0$ be any maximal by inclusion (non-uniform) eventown which contains $\mathcal{E}$; clearly $C_0$ is a linear subspace over $\mathbb{F}_2$.
Then 
\[
\{e + c_i \,|\, 1\leq i \leq t(m-t) \}
\]
is a 4-uniform eventown $\mathcal{F}$ of size $t(m-t)$ which is contained in $C_0 \cap \mathcal H_e$ (here and below ``+'' means addition over $\mathbb{F}_2$ which corresponds to the symmetric difference in the set theoretical language).
Now let us consider maximal with respect to inclusion 4-uniform eventown $\mathcal{G}$ containing $\mathcal{F}$. 
Let $C$ be any maximal by inclusion (non-uniform) eventown which contains $\mathcal{G}$; $C$ is also a linear subspace over $\mathbb{F}_2$. Since $\mathcal{G}$ is maximal by inclusion we have
\[
C \cap \binom{[n]}{4} = \mathcal{G}.
\]

Define $\mathcal{F}_7$ as a family of 4-sets with pairwise intersections of size 2 (the family of complements in $[7]$ of the lines of the Fano plane).
Define $\mathcal{F}_8$ to be the union of $\mathcal{F}_7$ and the complements of members of $\mathcal{F}_7$ in the set $[8]$.

\begin{table}[h]
    \centering
    \begin{tabular}{ccccccc}
     1 & 2 & 3 & 4 & & & \\
     1 & 2 & & & 5 & 6 & \\
     1 & & 3 & & 5 & & 7 \\
     1 & & & 4 & & 6 & 7 \\
     & 2 & 3 & & & 6 & 7 \\
     & 2 & & 4 & 5 & & 7 \\
     & & 3 & 4 & 5 & 6 &  
    \end{tabular}
    \caption{The definition of $\mathcal{F}_7$}
    \label{tab:my_label}
\end{table}

\begin{lemma} \label{lm:weight-four-decomposition}
    Let $\mathcal{G} \subset \binom{[n]}{4}$ be a maximal with respect to inclusion $4$-uniform eventown. 
    Then $\mathcal{G}$ is a union of an atomic construction $\mathcal{A}_\mathcal{G}$ and several copies of $\mathcal{F}_7$ and $\mathcal{F}_8$ with pairwise disjoint supports.
    Thus 
    \[
    |\mathcal{G}| = \binom{a}{2} + 7b + 14c,
    \]
    where $2a + 7b + 8c \leq n$.
\end{lemma}

\begin{proof}
By maximality we have $\spann (\mathcal{G}) \cap \binom{[n]}{4} = \mathcal{G}$.
Suppose that $\mathcal{G}$ is not an atomic eventown.

For every $i \in [n]$ let $T_i$ be a set of $j \in [n]\setminus \{i\}$ such that every quadruple in $\mathcal G$ containing $i$ also contains $j$, and let $\mathcal{Y}$ be a set of functions $f: [n] \rightarrow[n]$, such that $f(i) \in T_i$ for every $i$. 
If for every $i$ one has $T_i \neq \emptyset$, then $\mathcal{Y} \neq \emptyset$.

If there exists the involution $f \in \mathcal{Y}$, then we can split $[n]$ into pairs, which are orbits of $f$ (by definition $f$ has no fixed point), so $\mathcal{G}$ is subset of the atomic eventown on these pairs, so it is exactly atomic construction on these pairs. 

For every $f \in \mathcal{Y}$ define
$Bd(f) = \{i \in  [n] \,|\, f(f(i)) \neq i \}$.
$\mathcal{Y}$ contains no involution, so $Bd(f) \neq \emptyset$ for every $f \in \mathcal{Y}$. Let $f$ be an arbitrary function with the minimal value of $|Bd(f)|$.

Since $f$ is not an involution, there exist $i \in Bd(f),j,\ell \in [n]$  such that $f(i) = j$, $f(j) = \ell$, $i \neq \ell$. By the definition of $\mathcal{Y}$ and $T_i$ every quadruple from $\mathcal G$ containing $i$ also contains $j$ and $\ell$. Since $\mathcal G$ is an eventown this means that there is only one quadruple in $\mathcal{G}$ containing $i$, say, $\{i,j,\ell,r\}$.
Then for every $x \in\{i,j,\ell, r\}$ one has $T_{x} = \{i,j,\ell, r\} \setminus \{x\}$.
Then we define function $g: [n] \rightarrow [n]$ as
\[
g(x)= \begin{cases}
f(x), x \neq j, \ell, r \\
i, x=j \\
r, x=\ell \\ 
\ell, x=r. \\
\end{cases}
\]
Obviously $g \in \mathcal{Y}$ and 
$Bd(g)  \subset Bd(f) \setminus \{i,j,\ell, r\}$ because for $x \in \{i,j,\ell,r\}$ one has $g(g(x))= x$. Also, for $x\in [n] \setminus Bd(f) \setminus \{i,j,\ell, r\}$ one has  $f(f(x))= x$, so $f(x) \notin \{i,j,\ell, r\}$, so $x = f(f(x)) = f(g(x)) = g(g(x))$. Then $x \notin Bd(g)$. 
The element $i$ is a witness of inequality $|Bd(g)| < |Bd(f)|$ which is a contradiction.

Then for some element $x \in [n]$ one has $T_{x} =\emptyset$.
Thus $\mathcal{G}$ contains sets $e_1 = \{x,y_1,y_2,y_3\}$, $e_2 = \{x,y_1,y_4,y_5\}$, $e_3 = \{x,y_2,y_4,y_6\}$ for some $y_1,\dots, y_6 \in [n]$.
Hence $\{e_1, e_2, e_3, e_1 + e_2, e_1 + e_3, e_2 + e_3, e_1 + e_2 + e_3\}$ is a copy $E$ of $\mathcal{F}_7$ in $S := \{x,y_1,\dots, y_6\}$. 
If a set $g$ from $\mathcal{G} \setminus E$ intersects $S$ then $C$ contains $g + e$ for every $e \in E$. Hence $\spann(e_1,e_2,e_3,g) \cap \binom{[n]}{4}$ is a copy of $\mathcal{F}_8$ in $\mathcal{G}$ with support $S \cup \{y_7\}$.
Excluding copies of $\mathcal{F}_7$ and $\mathcal{F}_8$ one by one we obtain a (possibly empty) atomic construction $\mathcal{A}_\mathcal{G}$.
\end{proof}

We now estimate the number of members of $\mathcal G$ that meet $e$
in exactly two points. 
 
By Lemma~\ref{lm:weight-four-decomposition}, the family $\mathcal G$
is a union, on pairwise disjoint coordinate supports, of an atomic part $\mathcal {A}_\mathcal{G}$ and copies of $\mathcal F_7$ and $\mathcal F_8$.

Let $A$ be the number of atoms of $\mathcal {A}_\mathcal{G}$ contained in $e$, $B$ be the number of atoms of $\mathcal {A}_\mathcal{G}$ contained in $[n] \setminus e$
and $R$ be the number of split atoms, that is, atoms containing one point of $e$ and one point of $[n] \setminus e$.
Let $S_1,\ldots,S_h$ be the supports of the exceptional components,
that is, the exceptional components whose intersection with $\mathcal H_e$ is nonempty. Put
\[
    q
    :=
    \sum_{j=1}^{h}|S_j\cap e|,
    \qquad
    w
    :=
    \sum_{j=1}^{h}|S_j\cap ([n] \setminus e)|.
\]
We have $|S_j\cap e|\geq 2$ and $|S_j|\leq 8$, so
\[
    |S_j \cap ([n] \setminus e)|
    \leq
    3|S_j\cap e|.
\]
Consequently, $w\leq 3q$ and $2h\leq q$.
Counting the coordinates of the atomic and exceptional components
inside $e$ and $[n] \setminus e$, respectively, gives
$2A + R + q\leq 2t$ and $2B + R + w \leq 2(m-t)$.

A member of the atomic component  $\mathcal {A}_\mathcal{G}$ belongs to
$\mathcal H_e$ precisely in one of the following two cases:
\begin{enumerate}
    \item it is the union of one atom contained in $e$ and one atom
          contained in $[n] \setminus e$;
    \item it is the union of two atoms meeting both $e$ and $[n] \setminus e$.
\end{enumerate}
Therefore
\[ 
| \mathcal {A}_\mathcal{G}\cap\mathcal H_e| \leq AB+ \binom{R}{2}.
\]

Each exceptional component contains at most $14$ sets, hence in total at most $14h\leq 7q$ members of $\mathcal{G}$ come from the exceptional components.
Summing up
\[
    |\mathcal G\cap\mathcal H_e|
    \leq
    AB+\binom R2+7q
    \leq
    \left(t-\frac{q+R}{2}\right)
    \left(m-t-\frac{w+R}{2}\right)
    +\binom R2+7q.
\]

Set $u:=q+R$ and $v:=w+R$, so that $v=w+R\leq 3q+3R=3u$ and $u\leq 2t$.
Now consider the difference between the lower bound coming from the distance distribution and the above upper bound on $|\mathcal G\cap\mathcal H_e|$ 
\[
t(m-t) - \left(
        \left(t-\frac u2\right)
        \left(m-t-\frac v2\right)
        +\binom R2+7q
    \right)
=
    \frac{tv}{2}
    +\frac{u(m-t)}{2}
    -\frac{uv}{4}
    -\binom R2
    -7q
    \geq
    \]
    \[
    \frac{u(m-t)}{2}
    -\frac{3u^2}{4}
    -\frac{u^2}{2}
    -7u
    =
    u\left(
        \frac{m-t}{2}
        -\frac{5u}{4}
        -7
    \right)
    \geq
    u\left(
        \frac{m-t}{2}
        -\frac{5t}{2}
        -7
    \right)
    =
    \frac{u}{2}(m-6t-14).
\]
Since $m>6t+14$, the preceding inequalities force $u=0$, hence $q=R=0$, and therefore $w = 0$.  
Hence there are no components isomorphic to $\mathcal F_7$ or $\mathcal F_8$, and no
atom of an atomic component is split between $e$ and $[n] \setminus e$.

Thus $\mathcal{F} = \mathcal G \cap \mathcal H_e$ and also $\mathcal{F}$ coincides with the intersection of some atomic construction $\mathcal A$ with $\mathcal H_e$.
Since $\mathcal F, \mathcal E \subset C_0$, we have $\mathcal{A} \cap \binom{[n]}{4} \subset C_0$.
Suppose that the intersection of some $k$-set $e_1 \in \mathcal E$ with some atom $p_1$ of $\mathcal A$ has size one. Since $m > 2t$ there is an atom $p_2$ of $\mathcal A$ which does not intersect $e_1$. This is a contradiction since $p_1 \cup p_2$ belongs to $C_0$. 

Thus every member of $\mathcal E$ meets every atom $p_i$ in an even number of points. Consequently, $\mathcal{E} \subset \mathcal{A}$ and inherits atomic structure.

\begin{remark}
The present setup is reminiscent of recent work on $T$-avoiding
spherical codes~\cite{boyvalenkov2025kissing,boyvalenkov2026universal,boyvalenkov2025energy,boyvalenkov2026t}.
Many of the examples considered there arise as minimum-vector layers
of extremal lattices.  This is analogous to the relation between the
uniform atomic eventown and the full atomic eventown: the latter is
a self-dual binary code, while the former is one of its
constant-weight layers.

The classification problem is substantially less rigid in the
spherical setting.  In dimension $32$, there are more than $10^7$
non-isometric extremal even unimodular lattices, and every such lattice produces an extremal $T$-avoiding kissing configuration.  
Such lattices are generated by their minimal vectors, so their normalized
minimum-vector layers give more than $10^7$ non-isometric spherical
codes.  Nevertheless, all these layers have the same distance
distribution.  Thus the distance distribution alone cannot
distinguish, let alone classify, the extremal configurations.
\end{remark}

\paragraph{Acknowledgments.} We thank Frank Vallentin for the discussion around the key lemma.
Danila Cherkashin is supported by Bulgarian NSF grant KP-06-N72/6-2023.

\bibliographystyle{plain}
\bibliography{main}

\end{document}